\documentclass[10pt]{csarticle}

\usepackage[numbers,sort&compress]{natbib}
\usepackage{pdfsync}
\usepackage{hyperref}

\newcommand{\scr}[1]{\mathcal{#1}}

\newcommand{\spt}{\mathrm{SPT}}
\newcommand{\expo}{\mathrm{Exp}}

\newcommand{\psub}[1]{\mathbf{P}_{\epsilon}\left\{#1\right\}}
\newcommand{\Expn}[2]{\mathbf{E}_{#1}\event{ #2 }}

\newcommand{\Po}{\mathrm{Po}}

\hypersetup{
    bookmarks=true,         
    unicode=false,          
    pdftoolbar=true,        
    pdfmenubar=true,        
    pdffitwindow=true,      
    pdftitle={My title},    
    pdfauthor={Author},     
    pdfsubject={Subject},   
    pdfnewwindow=true,      
    pdfkeywords={keywords}, 
    colorlinks=true,       
    linkcolor=blue,          
    citecolor=blue,        
    filecolor=blue,      
    urlcolor=blue           
}

{ \end{list} }

\begin{document}

\title{The longest minimum-weight path in a complete graph}
\author{Louigi Addario-Berry\and Nicolas Broutin \and G\'abor Lugosi
\thanks{
    Email: \tt louigi@gmail.com, nicolas.broutin@m4x.org, lugosi@upf.es
    \rm . 
}
}
\maketitle

\begin{abstract}
We consider the minimum-weight path between any pair of nodes
of the $n$-vertex complete graph in which the weights of the
edges are i.i.d.\ exponentially distributed random variables. 
We show that the longest of these minimum-weight paths has about
$\alpha^\star \log n$ edges where $\alpha^\star\approx 3.5911$ is the unique
solution of the equation $\alpha \log \alpha - \alpha =1$. 
This answers a question posed by Janson \cite{janson99one}. 
\end{abstract}


\section{Introduction}
We consider the complete graph $K_n$ on $n$ vertices $[n]:=\{1,\ldots,n\}$, 
augmented with independent exponential mean $n$ (or $\expo(n)$ for short) edge weights $\{X_e:e\in E(K_n)\}$. 
For any subgraph $H=(V(H),E(H))$ of $K_n$ we write $|H|$ for $|E(H)|$ and let
\[
w(H) = \sum_{e \in E(H)} X_e.
\]
For $i,j \in [n]$ we let $P_{ij}$ be the minimum-weight path from $i$
to $j$ and write $W_{ij}=w(P_{ij})$ (adopting the convention
$P_{ii}=\emptyset$, $w_{ii}=0$).  For any fixed vertex~$i$,
$\cup_{j\neq i} P_{ij}$ is a tree, the {\em shortest path tree}
$\spt_i$ rooted at $i$.  For $t \geq 0$ and $k\in [n]$, we let
$\spt_i(t)$ be the subtree of $\spt_i$ induced by nodes $j$ with
$W_{ij} \leq t$.

Let $t^i_k$ be the first time for which $\spt_i(t)$ contains at least $k+1$ vertices, and write $t^i_0=0$. Due to the memoryless property of the exponential distribution, for
each $k\in [n]$, the location of attachment of the vertex added at
time $t_k^i$ is uniform among the vertices of $\spt_i(t^i_{k-1})$; in
other words, $\spt_i(t^i_k)$ is distributed like a {\em random recursive tree} with $k$
vertices. Random recursive trees have been well-studied \cite{smythe95rec}; in 
particular, it is known that the depth (number of edges on the path
from a uniformly random node to the root) in a recursive tree with $n$
vertices is asymptotic to $\log n$ in probability,
see Devroye \citep{devroye87branching}, and the height (greatest number of edges
on any path starting from the root) is asymptotic to $e\log n$ in
probability, see Devroye \citep{devroye87branching} and 
Pittel \citep{pittel94recursive}.

A variety of authors have studied the {\em weighted} structure of a single tree
$\spt_1$ (or $\spt_i$ for any $i$), and of the family of shortest
paths $\{P_{ij}:i,j \in [n]\}$.  Van der Hofstad, Hooghiemstra, and
van Mieghem~\cite{vanderhofstad2006saw,vanderhofstad2007wsp} prove
that, letting $W$ be the total weight of $\spt_1$,
\[
\frac{(W - \zeta(2))}{\sqrt{n}}\xrightarrow[n\to\infty]{}\scr{N}(0,4\zeta(3))
\]
in distribution. Here, $\zeta$ denotes the Riemann zeta function $\zeta(s):=\sum_{k\ge 1} k^{-s}$. Janson \cite{janson99one} proved the beautiful result that
\[
\frac{W_{12}}{\log n}   \xrightarrow[n\to\infty]{} 1, \quad
\frac{\max_{1 \leq j \leq n}W_{1j}}{\log n}   \xrightarrow[n\to\infty]{} 2,  \quad \mbox{ and } \quad
\frac{\max_{1 \leq i,j \leq n} W_{ij}}{\log n}  \xrightarrow[n\to\infty]{}  3
\]
in probability,
and provided more detailed information about the distribution of $W_{12}$ and of $\max_{1 \leq j \leq n}W_{1j}$. (Janson's results 
in fact hold for some edge weight distributions other than exponential, see Section~\ref{sec:concl}.) 
On the other hand, the unweighted structure of shortest path trees 
has been studied by \citet{hooghiemstra2008} who derived the limiting distribution of the pair $(W_{12}, |P_{12}|)$ suitably rescaled. 

It is immediate from the aforementioned results for the depth of nodes and the height in random recursive trees that 
\[
\frac{|P_{12}|}{\log n} \xrightarrow[n\to\infty]{} 1 \qquad \mbox{ and } \qquad \frac{\max_{1 \leq j \leq n}|P_{1j}|}{\log n} \xrightarrow[n\to\infty]{}  e
\]
in probability. Furthermore, as noted by \citet{janson99one}, the tail bounds for the height of random recursive trees established by \citet{devroye87branching} imply that
\begin{equation}\label{eq:devroye}
\limsup_{n \rightarrow \infty} \frac{\max_{1 \leq i,j \leq n} |P_{ij}|}{\log n}   \leq  \alpha^\star \quad\mbox{ in probability},
\end{equation}
where $\alpha^\star\approx 3.5911$ is the unique solution of $\alpha\log\alpha - \alpha
= 1$.  Janson \cite{janson99one} asked whether
there is a constant $c \in [e,\alpha^\star]$ such that 
$\max_{1 \leq i,j \leq n}|P_{ij}|/\log n\rightarrow c$ in probability. It is the
purpose of this note to answer Janson's question in the affirmative:
\begin{thm}\label{thm:main}
For any function $\omega(n)$ tending to $+\infty$ with $n$, for some constant $L>0$, and for all $\delta \in (0,1)$, 
for all $n$ sufficiently large, with probability greater than $1-\delta$,
\[
\alpha^\star \log n - L \log\log n \le \max_{1 \leq i,j \leq n}|P_{ij}| \leq \alpha^\star \log n + \omega(n)~. 
\]
Furthermore, 
\[
\alpha^\star \log n - L \log \log n \leq \E{\max_{1 \leq i,j \leq n}|P_{ij}|} \leq \alpha^\star \log n + L. 
\]
\end{thm}
Of course, it follows immediately from the first part of Theorem \ref{thm:main} that $\max_{1 \leq i,j \leq n}|P_{ij}|/\log n\rightarrow \alpha^\star$ in probability. 
The easy half of the result is the upper bound as (\ref{eq:devroye})
and other results from Devroye \cite{devroye87branching} will make
proving the upper bounds of the theorem a routine exercise.  The
nontrivial part of the result is the lower bound
 which we prove using
the second moment method, applied to a suitably defined set of
shortest paths with special properties that make them amenable to
analysis.  Before we get going, a brief remark is in order: we have to
work with exponential edge weights in order to exploit the memoryless
property of the exponential distribution.  In the conclusion
(Section~\ref{sec:concl}) we explain how to extend our results to a
range of other edge weights
that have a finite and positive density at $0$,
including uniform edge weights.

\section{A useful lemma and the upper bound}
For $i \in [n]$ and $k \in [n-1]$ let $\tau_k^i = t_{k+1}^i - t_{k}^i$ denote the $k$-th interarrival time. By the memoryless property of the exponential distribution, $\{\tau^i_k, 1\le k <n\}$ are independent. Furthermore,
$\tau_k^i$ is the minimum of $k(n-k)$ independent exponential random variables with mean $n$, hence $\tau^i_k$ has an exponential distribution with rate parameter $k(n-k)/n$.
Using this fact, the following lemma provides upper bounds on the size of $\spt_i(t)$. 

\begin{lem}\label{lem:sptheight}
For all $t \geq 0$ and integers $m \in [n-1]$ with $m \geq e^t$, 
\[
\p{|\spt_i(t)| \geq m} \leq 3\cdot \sqrt{m/e^t} \cdot e^{- m/e^t}.
\]
\end{lem}
\begin{proof}Observe that, for any $c>0$, 
\[
\p{|\spt_i(t)| \geq m} 
 =  \p{\sum_{k=1}^m \tau^i_k \leq t}=  \p{\prod_{k=1}^m \exp\pran{ -c\tau_k^i} \geq e^{-ct}}.
\]
By Markov's inequality and the independence of the interarrival times  $\tau_1^{i},\ldots,\tau_{n-1}^i$, we have
\begin{eqnarray*}
\p{|\spt_i(t)| \geq m}  \leq  e^{ct} \prod_{k=1}^m \E{ \exp\pran{ -c\tau_k^i } } 
 \le  e^{ct} \prod_{k=1}^m \frac{k}{k+c} \le e^{ct} \cdot \frac{\Gamma(c+1)}{m^c},
\end{eqnarray*}
where $\Gamma$ is the Gamma function.
Using the strengthening of Stirling's formula 
$\Gamma(c+1) \le \sqrt{2\pi c} (c/e)^c e^{1/12}$ (see, e.g.,~\cite{robbins1955rss}), and setting $c=m/e^t$, we obtain
\[
\p{|\spt_i(t)| \geq m}  \leq \sqrt{2\pi} e^{1/12} \cdot \sqrt c \cdot e^{-c} \le 3 \cdot \sqrt c \cdot e^{-c},
\]
as claimed.
\end{proof}
We will also use the following result from Devroye \cite{devroye87branching}, providing an upper bound on the heights of random recursive trees. In the following, $h(T)$ denotes the height of a rooted tree $T$.
\begin{thm}[\cite{devroye87branching}, Theorem 9]\label{thm:luc}
Let $T_m$ be a random recursive tree with $m\geq 2$ nodes. Then, for $x>1$,
\[
\p{h(T_m) \geq x\log m} \leq e^{x-1} m^{x - x\log x}.
\]
\end{thm}
The upper bounds in Theorem \ref{thm:main} follow immediately from Theorem \ref{thm:luc}. 
\begin{proof}[Proof of Theorem \ref{thm:main}, upper bounds] 
For any $t\ge 0$,
by a union bound we have, for $n>2$, 
\[
\p{\max_{1 \leq i,j \leq n}|P_{ij}| \geq \alpha^\star\log n + t} 
\le n \p{\max_{1 \leq j \leq n}|P_{1j}| \geq \alpha^\star\log n + t}.
\]
The event in the right-hand side above holds if $\spt_1$ has
height at least $\alpha^\star \log n+t$. 
Since $\spt_1$ is distributed like a random recursive tree on $n$
nodes, to further bound the probability on the right-hand side,
we may use Theorem~\ref{thm:luc} with $m$ replaced by $n$ and
$x= \alpha^\star + t/\log n$. In order to simplify the obtained
expression,
observe also that, by the
definition of $\alpha^\star\ge e$, for $t\ge 0$,
\[
x\log x -x = \left(\alpha^\star+\frac{t}{\log n}\right) 
 \log\left(\alpha^\star+\frac{t}{\log n}\right)  - \left(\alpha^\star+\frac{t}{\log n}\right) \ge 1+\frac{t}{\log n}~.
\]
We obtain that
\begin{eqnarray*}
\p{\max_{1 \leq j \leq n}|P_{1j}| \geq \alpha^\star\log n + t}
\le
e^x n^{x-x\log x} 
\le 
e^{\alpha^\star + t/\log n} n^{-1+ t/\log n}
\end{eqnarray*}
and therefore
\begin{equation}\label{eq:upper_max}\p{\max_{1 \leq i,j \leq n}|P_{ij}| \geq \alpha^\star\log n + t} \le e^{\alpha^\star + t/\log n}
e^{-t}~.\end{equation}
Choosing $t=t(n)\to\infty$ proves the upper bound in probability. 
Now, using the bound (\ref{eq:upper_max}), we have, for $n\ge 8$,
\begin{eqnarray*}
\E{\max_{1 \leq i,j \leq n}|P_{ij}|} 
&\leq& 
\alpha^\star\log n +1+ 
 \sum_{\ell=1}^{\infty}2^\ell \p{\max_{1 \leq i,j \leq n}|P_{ij}|\geq \alpha^\star\log n + 2^{\ell-1}} \\
&\leq& \alpha^\star\log n +1+ e^{\alpha^\star} \cdot \sum_{\ell=1}^{\infty} 2^\ell \exp(-2^{\ell-2})\\
& = & \alpha^\star\log n + O(1),
\end{eqnarray*}
proving the upper bound in expectation. 
\end{proof}

\section{Towards the lower bound}\label{sec:notesforlb}
Fix a path $P=v_1,\ldots,v_{k+1}$, $k \leq 12\log n$ (where 12 is somewhat arbitrary, 
the key point being that it's bigger than $3\alpha^\star$). 
It is easily seen that for $w > 0$, letting $\Po(w)$ denote a Poisson mean $w$ random variable, we have 
\[
	\p{w(P) \leq wn} = \p{\Po(w) \geq k} = \sum_{i=k}^{\infty} \frac{w^ke^{-w}}{k!}.
\]
When $w = o(1)$, the above sum is dominated by its first term and $e^{-w}=1-o(1)$, so we have
\begin{equation}\label{eq:poisson_approx}
\p{w(P)\leq wn} = (1+o(1)) \frac{w^k}{k!}.
\end{equation}
We next show that {\em given} that $P$ has small weight, it is very likely to be the minimum-weight path 
between its endpoints. More precisely, let $P^\star(v,w)$ be the minimum-weight 
path in $K_n$ between two vertices $v$ and $w$. Then we have the following.
\begin{lem}\label{lem:light_means_lightest}
Fix $c>2$. For all $n$ sufficiently large, for all $\epsilon > c \log\log n/\log n$ and any path $P=v_1,\ldots,v_{k+1}$ in $K_n$ with $1\le k \leq 12\log n$, we have 
\[
\probC{P\neq P^\star(v_1,v_{k+1})}{w(P) \leq (1-\epsilon)\log n} \leq \frac{13 k^2}{n^{\epsilon}}.
\]
\end{lem}
The restriction $c>2$ is not necessary in the above lemma, but the upper bound becomes trivial when $\epsilon \leq 2\log\log n/\log n$. 
\begin{proof}
Let $\spt_v(t)$ denote the shortest path tree started from a vertex $v$ and stopped at time $t$. For $k \geq 2$, if 
$P\neq P^\star(v_1,v_{k+1})$ then it must be the case that for some $i=1,\ldots,k$, in $K_n\setminus E(P)$, 
$\spt_{v_i}((1-\epsilon)\log n)$ contains one of $v_{i+1},\ldots,v_{k+1}$. 

Allowing connections using the edges along the path $P$ only increases the probability of this event.
Let $K_n^P$ be $K_n$ where the edge weights $X_e$, $e\in E(P)$, along the path $P$ have been replaced by independent copies. Let $A_i$ be the event that, in $K_n^P$,
$\spt_{v_i}((1-\epsilon)\log n)$ contains one of $v_{i+1},\ldots,v_{k+1}$ for some $i=1,\ldots,k$. The remark above then implies by the union bound that 
\begin{equation}\label{eq:bound_up}
\probC{P\neq P^\star(v_1,v_{k+1})}{w(P) \leq (1-\epsilon)\log n} \le \sum_{i=1}^k \p{A_i}.
\end{equation}

Let $N_i$ denote the number of nodes in $\spt_{v_i}((1-\epsilon)\log n)$ in $K_n^P$. 
Observe that, conditioning on $N_i$, the probability that 
$A_i$ does {\em not} occur is
\begin{eqnarray}\label{eq:pai}
1-\Cprob{A_i}{N_i}	&=& \pran{\prod_{j=1}^{N_i} \pran{1 - \frac{k-i+1}{n-j}}}\cdot\I{n-N_i > k-i+1} \nonumber\\
				&\geq& \max\pran{1 - \frac{(k-i+1)N_i}{n-N_i},0}, 
\end{eqnarray}
where $\I{\,\cdot\,}$ denotes the indicator function. 
Let $E_i(0)$ be the event that $N_i \leq 2n^{1-\epsilon}$, for $j=1,\ldots,\lfloor \log_2 (n^{\epsilon}/2k)\rfloor-1$, 
let $E_i(j)$ be the event that $2^jn^{1-\epsilon} < N_i \leq 2^{j+1}n^{1-\epsilon}$, and let 
$E_i(\lfloor \log_2 (n^{\epsilon}/2k)\rfloor)$ be the event that $N_i > 2^{\lfloor \log_2 (n^{\epsilon}/2k)\rfloor}n^{1-\epsilon}$. 
Expressing $\p{A_i}$ as a sum of conditional probabilities, we have
\begin{eqnarray}\label{eq:sum_cond}
\p{A_i} 	
& = & \sum_{j=0}^{\lfloor \log_2 (n^{\epsilon}/2k)\rfloor}\probC{A_i}{E_i(j)}\cdot \p{E_i(j)} \\
& \leq & \sum_{j=0}^{\lfloor \log_2 (n^{\epsilon}/2k)\rfloor}\probC{A_i}{E_i(j)} \cdot 3\sqrt{2^j}e^{-2^j},\nonumber
\end{eqnarray}
by Lemma~\ref{lem:sptheight}. Using (\ref{eq:pai}) to bound $\Cprob{A_i}{E_i(j)}$,
\begin{eqnarray}
\p{A_i}				
& \leq & \sum_{j=0}^{\lfloor \log_2 (n^{\epsilon}/2k)\rfloor-1}\frac{(k-i+1)2^{j+1}n^{1-\epsilon}}{n-2^{j+1}n^{1-\epsilon}}\cdot \frac{3\sqrt{2^j}}{e^{2^j}} + \frac{3\sqrt{2^{\lfloor \log_2 (n^{\epsilon}/2k)\rfloor}}}{e^{2^{\lfloor \log_2 (n^{\epsilon}/2k)\rfloor}}}\nonumber\\
& \leq & \frac{6(k-i+1)}{n^{\epsilon}(1-1/2k)}\sum_{j=0}^{\infty}\frac{2^{3j/2}}{e^{2^j}} + \frac{3\sqrt{2^{\lfloor \log_2 (n^{\epsilon}/2k)\rfloor}}}{e^{2^{\lfloor \log_2 (n^{\epsilon}/2k)\rfloor}}} \nonumber\\
& \leq & \frac{12(k-i+1)}{n^{\epsilon}} + n^{-\epsilon}, 
\end{eqnarray}
where we have used the fact that $\sum_{j\ge 0}2^{3j/2} e^{-2^j}\le 1$ and
the last inequality holds for all $k \leq 12 \log n$ and $n$ sufficiently large. It follows that, since $k\ge 1$,
\[
\sum_{i=1}^k \p{A_i} \leq \frac 1 {n^\epsilon} \pran{12 \binom{k+1}2 + k} \le \frac{13 k^2}{n^{\epsilon}},
\]
which, together with (\ref{eq:bound_up}),
completes the proof.
\end{proof}
The following fact follows immediately \emph{from the proof} of Lemma \ref{lem:light_means_lightest} (there are just fewer terms in the sum (\ref{eq:sum_cond})); it will be useful later.
\begin{cor}\label{cor:light_means_lightest}
Fix $c>2$. For all $n$ sufficiently large, for all $\epsilon > c \log\log n/\log n$ and any path $P=v_1,\ldots,v_{k+1}$ in $K_n$ with $1\le k \leq 12\log n$, we have 
\[
\probC{P\neq P^\star(v_1,v_{k+1})}{(1-2\epsilon)\log n \leq w(P) \leq (1-\epsilon)\log n} \leq \frac{13 k^2}{n^{\epsilon}}.
\]
\end{cor}
It follows from Lemma \ref{lem:light_means_lightest}, at least intuitively, that to get a lower bound of $k$ on the length of the longest minimum-weight path, 
we can instead bound from below the probability that there is {\em some} path with $k$ edges of weight $(1-\epsilon)\log n$. 
Given a positive integer $k$ and real $\epsilon>0$, let $\scr{P}_{k,\epsilon}=\scr{P}_{k,\epsilon}(n)$ be the set of paths with $k$ edges and weight at most $(1-\epsilon)\log n$ in $K_n$. 
As a first step, we remark that by (\ref{eq:poisson_approx}),  
\begin{equation}\label{eq:pkeps}
\e{|\scr{P}_{k,\epsilon}|} \sim n^{k+1}\pran{\frac{(1-\epsilon)\log n}{n}}^{k}\frac{1}{k!} \sim \frac{n}{\sqrt{2\pi k}} \pran{\frac{e\log n}{k}}^k(1-\epsilon)^k.
\end{equation}
Let $\alpha_{\epsilon}$ be the solution of 
\begin{equation}\label{eq:def_aeps}
\alpha\log \alpha - \alpha(1+\log(1-\epsilon)) = 1,
\end{equation}
so $\alpha_{\epsilon}<\alpha^\star$ and $\alpha_{\epsilon} \rightarrow \alpha^\star > e$ as $\epsilon \rightarrow 0$. 
When $t = k-\alpha_\epsilon\log n = o(\sqrt{\log n})$, we then have 
\[
\e{|\scr{P}_{k,\epsilon}|} =  \frac{(1+o(1))}{\sqrt{2\pi\alpha_{\epsilon}\log n}} \pran{\frac{e}{\alpha_\epsilon}}^t(1-\epsilon)^t.
\]
Letting 
\begin{equation}\label{eq:def_keps}
k_{\epsilon} = \alpha_{\epsilon}\log n - \beta_{\epsilon}\log\log n \qquad \mbox{where} \qquad \beta_\epsilon = \frac 1{2\pran{1+\log\alpha_\epsilon - \log(1-\epsilon)}},
\end{equation}
the following estimate for $\e |\scr{P}_{k, \epsilon}|$, for $k$ close to $k_\epsilon$. 
\begin{lem}\label{eq:exact_formula}
Let $\alpha_\epsilon, \beta_\epsilon$ and $k_\epsilon$ be defined as in (\ref{eq:def_aeps}) and (\ref{eq:def_keps}) above. Given $\epsilon$ with $0 < \epsilon < 1/2$, there is a positive constant $c_{\epsilon}$ such that 
$
\e{|\scr{P}_{k_{\epsilon},\epsilon}|} = (1+o(1))c_{\epsilon}
$ 
and, for $|t| = o(\sqrt{\log n})$, 
\[
\e{|\scr{P}_{k_{\epsilon}+t,\epsilon}|} = (1+o(1))c_{\epsilon}\cdot \pran{\frac{e}{\alpha_\epsilon}}^t(1-\epsilon)^t.
\]
\end{lem}
In particular, it follows that if $t = t(n) \rightarrow -\infty$ then $\e{|\scr{P}_{k_{\epsilon}+t,\epsilon}(n)|} \rightarrow \infty$. 
To derive a lower bound on the probability that there exists such a path, we use the second moment method, 
and now introduce the version of it we require. 
Given any random set $\scr{S}$ of paths in $K_n$ and two paths $P$ and $Q$, let $q_{\scr{S}}(P,Q)$ be the 
probability that $P$ and $Q$ are both in $\scr{S}$, and let 
\begin{equation}\label{eq:smmcount}
\Delta(\scr{S}) = \sum_{P,Q} q_{\scr{S}}(P,Q),
\end{equation}
where the sum is over pairs $P,Q$ of distinct but intersecting paths in $K_n$. 
By Corollary 4.3.4 of \citet{alon00proba}, we then have 
\begin{equation}\label{eq:smm}
\p{|\scr{S}|=0} \leq \frac{1}{\e{|\scr{S}|}} + \frac{\Delta(\scr{S})}{(\e{|\scr{S}|})^2}.
\end{equation}
Given the preceding discussion, a natural choice for the set $\scr{S}$ would be $\scr{P}_{k_{\epsilon}+t,\epsilon}$, for some 
$t=t(n)$ tending to $-\infty$ with $n$. Unfortunately, for this choice of $\scr{S}$ and for the values of $t$ we wish to consider,  
the quantity $\Delta(\scr{S})$ is too large for (\ref{eq:smm}) to yield a useful bound. However, it is both useful and instructive to proceed as though 
this was our choice of $\scr{S}$, and see how far we can get. 

Given paths $P$ and $Q$ in $K_n$, let 
\[
q_{\epsilon}(P,Q) = \p{w(P) \leq (1-\epsilon)\log n,w(Q) \leq (1-\epsilon)\log n}. 
\]
Also, for $t\in \R$, let 
\begin{equation}\label{eq:delta}
\Delta_{t}=\Delta_t(n,\epsilon) = \sum_{P,Q} q_{\epsilon}(P,Q)
\end{equation}
where the sum is over pairs $P,Q$ of distinct but intersecting paths with $\lceil k_{\epsilon}+t\rceil$ edges in $K_n$. 
(So, $\Delta_{t}$ is just $\Delta(\scr{S})$ when $\scr{S}$ is the set of paths in $\scr{P}_{\lceil k_{\epsilon}+t\rceil,\epsilon}$.)

\subsection{Light intersecting paths}

In order to bound $\Delta_t$, we first decompose the sum in (\ref{eq:delta}). For a given real number $t$ and integers $i,j$ with 
$1 \leq j \leq i < \lceil k_{\epsilon}+t\rceil$, 
let $\Delta_{t,i,j}$ be the sum of $q_{\epsilon}(P,Q)$ over paths $P,Q$ with $\lceil k_{\epsilon}+t\rceil$ 
edges, such that $P$ and $Q$ share $i$ edges, and these $i$ edges form precisely $j$ connected components.  
Then, we have 
\begin{equation}\label{eq:delta_decomp}
\Delta_t 
= \sum_{1 \leq j \leq i < \lceil k_{\epsilon}+t\rceil} \Delta_{t,i,j}. 
\end{equation}
We now need to consider (a) the probability that two paths of $k$
edges (for $k$ near to $k_{\epsilon}$) sharing $i$ edges both have
weight at most $(1-\epsilon)\log n$, and (b) the number of such pairs
of paths. The following lemma bounds the former probability.  Counting
the number of terms of a given sum $\Delta_{t,i,j}$ is the subject of
Section \ref{sec:intersection}.

\begin{lem}\label{lem:sharededges}
	Given two paths $P,Q$ in $K_n$, each consisting of $k$ edges, $i$ of which are common to $P$ and $Q$, and any $s\ge0$, we have, 
	$$\p{w(P)\le sn, w(Q)\le sn} \le 4^{k-i} \cdot \frac{s^{2k-i}}{(2k-i)!}.$$
\end{lem}
\begin{proof}
To make the formulas easier to read, introduce $w'(P)=w(P)/n$ and
$w'(Q)=w(Q)/n$. This corresponds to the case of exponential edge 
weights with mean $1$ instead of the
exponential mean $n$ for the individual edge weights. Thus, we need to
evaluate $\p{w'(P)\le s, w'(Q)\le s}$. Observe that the sum of $\ell$
exponential random variables is has a Gamma$(1,\ell)$ distribution,
with density function $f_\ell(t)=t^{\ell-1} e^{-t}/(\ell-1)!$ and
distribution function $F_\ell(t)$. So, conditioning on the aggregated
weight of the $i$ shared edges, we see that
$$\p{w'(P)\le s, w'(Q)\le s}=\int_0^s f_i(t) F_{k-i}(t)^2 dt.$$
However, for $t\ge 0$, $f_\ell(t)\le t^{\ell-1}/(\ell-1)!$ and $F_\ell(t)\le t^\ell/\ell!$, which implies that
$$\p{w'(P)\le s, w'(Q)\le s}\le \int_0^s \frac{t^{i-1}}{(i-1)!} \frac{(s-t)^{2(k-i)}}{(k-i)!^2} dt= \frac{s^{2k-i}}{(2k-i)!} \binom{2(k-i)}{k-i}.$$
Using the classical bound for the central binomial coefficients $\binom{2n}n\le 4^n$ completes the proof.
\end{proof}

\subsection{The number of intersecting pairs}\label{sec:intersection}

What is actually needed is to count the number of \emph{pairs} $(P,Q)$, where $P$ and $Q$ are two paths containing $k$ edges such that $P\cap Q$ has $i$ edges in $j$ connected components. 
More precisely, $P\cap Q$ is a graph composed of $j$ disjoint paths,
with $i$ edges in total.
Our aim is to analyze shortest paths in the graph, and hence it suffices to consider \emph{self-avoiding} paths (that do not intersect themselves even at vertices). Let $N_{k,i,j}$ denote the number of such pairs. 

\begin{lem}\label{lem:counting_pairs}
The number of pairs of self-avoiding paths $(P,Q)$ of length $k$ such
that $P\cap Q$ contains $i$ edges in $j$ connected components
satisfies
\[N_{k,i,j} \le n^{2k+2-i-j} (2k^3)^j.\]
\end{lem}
\begin{proof}Observe first that if $k < i+2j-2$ then $N_{k,i,j}=0$, since for a given path $P$, if $Q$ shares $i$ edges in $j$ connected components with $P$ then $Q$ has at least $i+2j-2$ edges. We now assume that $k \ge i+2j-2$.
	
We first focus on the choice of $P$ with the edges of $P\cap Q$
distinguished.  First fix a path $P$, self-avoiding, as an ordered  sequence of
$k+1$ vertices. There are $\binom{n}{k+1}\cdot (k+1)!$ such
choices. (In fact, we are double counting here, as the reversed sequence would yield the same path $P$; we can afford to ignore this fact in obtaining our upper bound.) 
We next choose the sequence of sizes of the $j$ parts of $P\cap Q$ as they appear along $P$. 
Since $P\cap Q$ contains $i$ edges, it suffices to split $i$
into $j$ ordered parts; there are $\binom{i+j}j$ possibilities for
this partition. Now that we have the sequence of sizes of the parts of
$P\cap Q$, it remains to choose the positions in $P$ where these edges
appear. We can count the number of such choices as follows: once we
have removed $P\cap Q$, there remains a sequence of $k-i$ edges of $P$. The
portions of $P\cap Q$ can be inserted at any of the $k-i+1$ separating
positions (the extremities are included), and hence there are
$\binom{k-i+1}j$ choices. So in the end, fixing the path $P$ together
with the edges of $P\cap Q$ can be done in
\begin{equation}\label{eq:counting_P}
\binom{n}{k+1} (k+1)! \cdot \binom{i+j}j \cdot \binom{k-i+1}j
\end{equation}
distinct ways.

It now remains to choose the second path $Q$, also self-avoiding, so
that it intersects $P$ at the distinguished edges. We first choose the
order in which the $j$ parts of $P\cap Q$ appear in $Q$: there are
$j!$ possible choices. The intersection $P\cap Q$ contains $i+j$ vertices and we need
$k+1-i+j$ other vertices to complete $Q$. It is possible that some
vertices of $Q\setminus (P\cap Q)$ are in $P$, so we can choose an
ordered sequence of these vertices in
$\binom{n-i-j}{k+1-i-j}(k+1-i-j)!$ different ways. 
(In fact, not all such sequences yield valid choices for $Q\setminus P$. For example, 
$Q\setminus P$ should not contain two consecutive vertices from $P$ in the same part, or $Q\cap P$ will not be what 
we claimed. However, we only seek an upper bound, and so can afford to ignore this issue.) 
Finally, we need to choose how the $j$
(now ordered) parts of $P\cap Q$ interlace with these $(k+1-i-j)$
extra vertices. The extra vertices define $k-i-j+1$ intervals (with
the extremities): there are $\binom{k-i-j+1}j$ ways to choose $j$ of
them (and insert the parts of $P \cap Q$ at these spots). It follows by this argument that for a fixed path $P$ with
distinguished edges forming $j$ parts, there are \emph{at most}
\begin{equation}\label{eq:counting_Q}
j! \cdot \binom{n-i-j}{k+1-i-j}(k+1-i-j)! \cdot \binom{k-i-j+1}j
\end{equation}
possible choices for the path $Q$. 

Now, the desired number $N_{k,i,j}$ of pairs $(P,Q)$ of paths of length $k$ such that $P\cap Q$ contains $i$ edges in $j$ connected parts is \emph{at most} the product of the numbers appearing in (\ref{eq:counting_P}) and (\ref{eq:counting_Q}). The desired bound follows by routine bounding
using the inequality $\binom{n}{k} \le n^k/k!$.
\end{proof}

\subsection{Paths intersecting at least twice}
Lemmas \ref{lem:sharededges} and \ref{lem:counting_pairs} yield more than sufficient control over all $\Delta_{t,i,j}$ with $j \geq 2$, which we quantify 
in the following lemma. 
\begin{lem}\label{eq:smm_bound}
There is an $\epsilon_0>0$ such that for all $0<\epsilon\leq \epsilon_0$, if $k=\alpha_{\epsilon}\log n + o(\sqrt{\log n})$, then for all $n$ sufficiently large 
\[
\frac{\sum_{2\leq j \leq i < k} \Delta_{t,i,j}}{\E{|\scr{P}_{k,\epsilon}|}^2} \leq n^{-0.95}.
\]
\end{lem}
\begin{proof} 
Combining the bound on the number of
intersecting paths (Lemma~\ref{lem:counting_pairs}) with Lemma \ref{lem:sharededges} with $s=(1-\epsilon)\log n /n$, we obtain
\begin{eqnarray*}
\Delta_{t,i,j} 
& := & N_{k,i,j} \cdot \p{w(P)\le sn, w(Q)\le sn}\\
&\le& n^{2-j} (2k^3)^j ((1-\epsilon)\log n)^{2k-i} \frac{4^{k-i}}{(2k-i)!}\\
&\le& n^{2-j} (2k^3)^j \pran{\frac{(1-\epsilon) e \log n}{2k-i}}^{2k-i} 4^{k-i},
\end{eqnarray*}
since, for $\ell \ge 1$, we have $\ell!\ge (\ell/e)^\ell$.
We have assumed that $j\ge 2$, so for $n$ large enough, $(2k^3)^jn^{2-j}\le 4 k^6$. It follows that
\begin{eqnarray*}
\sum_{2\leq j \leq i < k} \Delta_{t,i,j} 	
& \le & 	4k^8\max_{2 \leq i < k}\left\{4^{k-i} \pran{\frac{e(1-\epsilon)\log n}{(2k-i)}}^{2k-i}\right\}.
\end{eqnarray*}
Using the estimate (\ref{eq:pkeps}) for $\e |\scr{P}_{k,\epsilon}|$, it follows that, for $n$ large enough,
\begin{eqnarray}
\frac{\sum_{2\leq j \leq i < k} \Delta_{t,i,j}}{\E{|\scr{P}_{k,\epsilon}|}^2} 
&\leq& \frac{7 \pi k^9}{n^2}\max_{2 \leq i < k}\left\{\left(\frac{4e(1-\epsilon)\log n}{2k-i}\right)^{-i}\left(\frac{2k}{2k-i}\right)^{2k}\right\}.
\label{eq:jgeqtwo}
\end{eqnarray}
We bound the right-hand side of (\ref{eq:jgeqtwo}) by allowing $i$ to take real values.
We write 
$k = \gamma \log n$ and $x=\beta \log n$ (so $0 \leq \beta \leq \gamma$); we then see that 
\[
 \left(\frac{4e(1-\epsilon)\log n}{2k-x}\right)^{-x}
   \left(\frac{2k}{2k-x}\right)^{2k}
  = e^{g(\beta)\log n},
\]
where 
\[
g(\beta):=\beta\log \pran{\frac{2\gamma-\beta}{4e(1-\epsilon)}}
   + 2\gamma\log\pran{ \frac{2\gamma}{2\gamma -\beta}}.
\]
For $\gamma$ close to $\alpha^\star$, the function $g$ is suitably approximated by $g^\star$ defined by 
$$g^\star(\beta) := \beta \log\pran{\frac{2\alpha^\star-\beta}{4e}} + 2\alpha^\star
 \log\pran{ \frac{2\alpha^\star}{2\alpha^\star-\beta}}.$$ 
By differentiation, we see that $g^\star(\beta)$ is maximized on $[0,\alpha^\star]$ by taking $\beta = 2\alpha^\star-4$, 
at which point $g^\star(\beta) \approx 1.02$. 
For any $\delta$, we may ensure that $|\gamma-\alpha^\star| \leq \delta$ for $n$ large 
by choosing $\epsilon$ sufficiently small. 
Since $g(\beta)=g_{\gamma}(\beta)$ is bounded and continuous in both $\gamma$ and $\beta$ away from $\beta \in [2\gamma,\infty)$, 
for all $\epsilon$ sufficiently 
small and all $n$ sufficiently large we have
\[|g(\beta)-g^\star(\beta)| \leq 0.01\] 
for all $k= \alpha_{\epsilon}\log n + o(\sqrt{\log n})$ and 
all $\beta \in [0,\gamma]$. 
It follows that $g(\beta) \leq 1.04$ for all $\beta \in [0,\gamma]$, which, combined with 
(\ref{eq:jgeqtwo}), yields that 
\[
\frac{\sum_{2\leq j \leq i < k} \Delta_{t,i,j}}{\E{|\scr{P}_{k,\epsilon}|}^2} 
\le \frac{7\pi k^9}{n^{0.96}} = O\pran{\frac{\log^9 n}{n^{0.96}}}.
\]
This proves the lemma.
\end{proof}

\section{Dealing with paths intersecting only once}\label{sec:jequalsone}

Unfortunately, for paths intersecting only once (i.e., when $j=1$), for some values of $i$ the quantity $\Delta_{t,i,1}$ is too large for us to apply the straightforward approach 
used above. (Note that this is not an artefact of our upper bounds: Lemmas~\ref{lem:sharededges} and~\ref{lem:counting_pairs} are essentially tight up to logarithmic factors.
Also, two shortest paths between two pairs of vertices typically have one
connected component in common, so this class of paths is the main problem.)
The most natural and naive way to deal with this complication is to simply throw away all pairs of shortest paths $P$ and $Q$ 
whose intersection is connected, and try to bound the probability that one of the remaining minimum-weight paths is ``long'' (in the same sense 
as above). This is essentially our approach. However, in order to keep a handle on the conditioning imposed in doing so, it 
is useful to proceed ``from the other direction'': building a set of shortest paths with special properties that guarantee that (a) no pair of such paths 
has an intersection which is connected, yet (b) the set contains a minimum-weight path with about $\alpha^\star \log n$ edges.  

In order to describe this set, we first need to introduce a few concepts. 
We say that paths $P$ and $P'$ {\em intersect once} if $P \cap P'$ has only one connected component containing at least one edge.  (There may be other components of $P \cap P'$ which are isolated vertices.) 
We say that a path $P$ is a {\em local optimum} if, for all paths $P'$ with $|P|=|P'|$ that intersect $P$ once, 
we have $w(P') > w(P)$. From now on, let $\scr{O}_{k}$ denote the set of paths with $k$ edges that 
are local optima, let $\scr{O}_{k,\epsilon}$ denote the set of elements of 
$\scr{O}_k$ with weight at most $(1-\epsilon)\log n$.
Also, let $\scr{P}^\star=\scr{P}(n)=\{P_{ij}:1 \leq i < j \leq n\}$ denote the set of shortest paths in $K_n$. Note that 
$\scr{O}_{k,\epsilon} \subseteq \scr{P}_{k,\epsilon}$. 

\textsc{An instructive example.} To motivate our next definitions, consider the unlikely (impossible) but instructive event that we find a path $P=(v_1,\ldots,v_{k+1})$, 
all of whose edges have weight exactly $1$, and consider a path $Q$ of length $k$ intersecting $P$ in exactly $i < k$ consecutive edges ---say 
$Q\cap P = (v_1,\ldots,v_{i+1})$, for example. Let $Q_1$ (respectively~$Q_2$) be the component of $Q\setminus P$ containing \label{p:example}
$v_1$ (respectively~$v_{i+1}$). In order that $w(Q) \leq k$, then certainly we must have both $w(Q_1) \leq k-i$ and $w(Q_2) \leq k-i$. 
On the other hand, one of $Q_1$ and $Q_2$ has at least $(k-i)/2$ edges. It follows that if $w(Q) \leq k$ then in $K_n^P$ either 
$\spt_{v_1}(k-i)$ or $\spt_{v_{i+1}}(k-i)$ has height at least $(k-i)/2$. 

This observation is key to our approach. Roughly speaking, we wish to consider long minimum-weight paths with the special property 
that none of the shortest path trees leaving them are ``too tall''. This property will guarantee that the paths are local optima (Lemma~\ref{lem:key}), and we can 
then use the second moment method to prove concentration for the number of such paths. However, the idea suggested by the above sketch 
gives a little too much away in terms of what ``too tall'' means; in particular, it is not careful enough about the interplay between the contributions 
of $Q_1$ and $Q_2$ to the length of $Q$. When we formalize our idea, we will have to deal with this interplay to make the details work out.


\subsection{The fluctuations of conditioned partial sums}

In the above example, there was a second simplification: the edge weights have the extremely desirable
property that for every subpath $P'$ of $P$, $w(P')$ is {\em exactly}
$\CExp{w(P')}{w(P)=k}$, which makes the conditions we need to impose
on the shortest path trees leaving $P$ very easy to state.  It is, of
course, too much to ask for paths whose edge weights are as well
behaved as in the above example. In general, the weights of subpaths
will fluctuate considerably from their conditional expected value
given $w(P)$. Our tool for controlling the size of these fluctuations
is the {\em law of the iterated logarithm}; we will use the version
found in Rogers and Williams \cite{rogers00diffusions}. (These bounds
are in fact stronger than we require; however, we did not see a
substantially simpler proof of a simpler but sufficient result;
furthermore, Corollary \ref{cor:lilfin} is perhaps of interest
independent of its role in the current work.)

\begin{thm}[\citep{rogers00diffusions}, Corollary 16.5]
Let $Y_1,Y_2,\ldots$ be independent and identically distributed random variables with finite variance, and let $S_n = \sum_{j=1}^j Y_j$. Then
\[
\p{\limsup_{n \rightarrow \infty} \frac{S_n - n\e{Y_1}}{\sqrt{\V{Y_1}2n\log\log n}} = 1}=1.
\]
\end{thm}
We then immediately have the following finite statement, more useful for our purposes.
\begin{cor}\label{cor:lil}
Let $Y_1,Y_2,\ldots$ be independent and identically distributed random variables with finite variance, and let $S_n = \sum_{j=1}^j Y_j$. For all $\delta > 0$ there is a constant $C>0$ depending on $\delta$ and the distribution of $Y_1$ (we write $C(\delta, Y_1)$ for short) such that 
\[
\p{\sup_{n\ge 1}\left|\frac{S_n - n\e{Y_1}}{\sqrt{\V{Y_1}2n\log\log n}}\right| \leq C}\geq 1-\delta.
\]
\end{cor}
(To avoid annoying technicalities in our formulae, in Corollary \ref{cor:lil} and hereafter when we write $\log\log n$ we mean 
$\max(\log \log n,1)$.) 
At this point, the fact that we are considering exponential random variables comes in very handy, as it allows us to apply 
Corollary \ref{cor:lil} conditional upon the value of $S_n$. More precisely, we have 
\begin{cor}\label{cor:lilfin}
Suppose $Y_1$ is an exponential random variable. Then 
for all $\delta > 0$, there is $C'=C'(\delta)$ such that for all $n$ sufficiently large, 
\begin{equation}\label{eq:lilfin}
\Cprob{\forall k \in [n],~\left|\frac{S_k-S_n(k/n)}{\sqrt{\CVar{Y_1}{S_n}2k\log\log k}}\right| \leq C'}{S_n} \geq 1-\delta.
\end{equation}
\end{cor}
We emphasize that the probability in (\ref{eq:lilfin}) is a {\em random variable}, measurable with respect to $S_n$; 
the content of the lemma is that this random variable is {\em deterministically} at least $1-\delta$. 
In particular, if we let $\scr{E}$ be the event in (\ref{eq:lilfin})  (so (\ref{eq:lilfin}) bounds $\probC{\scr{E}}{S_n}$), 
then $\p{\scr{E}} = \Expn{S_n}{\probC{\scr{E}}{S_n}} \geq 1-\delta$ and more strongly, for any event 
$\scr{E}'$ which is measurable with respect to $S_n$, 
\[
\Cprob{\scr{E}}{\scr{E}'} = \CExpn{S_n}{\probC{\scr{E}}{S_n}}{\scr{E'}} \geq 1-\delta. 
\]
\begin{proof}[Proof of Corollary \ref{cor:lilfin}]
Fix $\delta>0$ and $n$.
For $i=1,\dots, n$, let $E_i = S_i/S_n$. 
Then the set $\{E_1,\ldots,E_n\}$ is independent of $S_n$ and distributed like $n$ independent $[0,1]$-uniform random variables (for a proof, see Chapter 8 of \cite{shorack86empirical}). It follows in particular that $E_1$, the minimum of $n$ independent uniforms, has a Beta$(1,n-1)$ distribution and $\CVar{Y_1}{S_n}=(S_n/n)^2(n-1)/(n+1)\ge (S_n/n)^2/3$, for $n\ge 2$.
Using this estimate of the conditional variance, (\ref{eq:lilfin}) may be
bounded as
\begin{eqnarray*}
\lefteqn{
\Cprob{\forall k \in [n],~\left|\frac{S_k-S_n(k/n)}{\sqrt{\CVar{Y_1}{S_n}2k\log\log k}}\right| \leq C'}{S_n} 
}  \\
& \geq &
\Cprob{\forall k \in [n],~\left|\frac{n\frac{S_k}{S_n}-k}{\sqrt{2k\log\log k}}\right| \leq \sqrt{3}C'}{S_n} \\
& = &
\prob{\forall k \in [n],~\left|\frac{n\frac{S_k}{S_n}-k}{\sqrt{2k\log\log k}}\right| \leq \sqrt{3}C'},
\end{eqnarray*}
where we used the independence of the $S_k/S_n$ of $S_n$.
Thus, proving (\ref{eq:lilfin}) reduces to finding $C'$ such that the 
\emph{unconditioned} statement 
\begin{equation}\label{eq:lilfin_aim}
\p{\sup_{1\le k\le n}~\left|\frac{(n/S_n)S_k-k}{\sqrt{2k\log\log k}}\right| \leq C'}\ge 1-\delta
\end{equation}
holds. By choosing $C'$ large enough we can certainly ensure that this
holds for small $n$; shortly we will take advantage of this fact.  By
Corollary \ref{cor:lil} we have that
\begin{equation}\label{eq:lilfin2}
\p{\sup_{1\le k \le n}~\left|\frac{S_k-k}{\sqrt{2k\log\log k}}\right| \leq C_1} \geq 1-\frac{\delta}{2},
\end{equation}
where $C_1=C(\delta/2,Y_1)$ as in the statement of Corollary \ref{cor:lil}. If, on the other hand, there is $k\in\{1,\dots, n\}$ such that 
\[
\left|\frac{(n/S_n)S_k-k}{\sqrt{2k\log\log k}}\right| \geq C_1+c,
\]
for some $c>0$, then either the event whose probability is bounded in (\ref{eq:lilfin2}) must fail to hold, or by the triangle inequality we must have
\[
\left|\frac{(n/S_n-1)S_k}{\sqrt{2k\log\log k}}\right| \geq c,
\]
for some $k$, or more simply $|S_n- n| \geq c\sqrt{2k\log\log k}\cdot {S_n}/{S_k}$.
For this to occur, we must either have $|S_n - n| \geq \sqrt{cn}$ or $S_k \geq \sqrt{c(2k\log\log k)}\cdot(\sqrt{n}-\sqrt{c})$. 
It follows that, for $n$ large enough that $\sqrt{n}-\sqrt{c} \geq \sqrt{n/2}$ (and recalling that $\log\log k \geq 1$ for all $k$ by convention), we have
\begin{eqnarray}\label{eq:large_dev}
\p{\sup_{1\le k\le n}\left|\frac{(n/S_n-1)S_k}{\sqrt{2k\log\log k}}\right| \geq c} 
&\le&  \sum_{k=1}^n \p{S_k \geq k \sqrt{c}} + \p{|S_n - n| \geq \sqrt{cn}}\nonumber \\
&\le&  \sum_{k=1}^n e^{-k H(\sqrt c)} + \frac 1 c, 
\end{eqnarray}
where we have bounded the first term using a large deviations bound for sums of exponential random variables \cite{dembo98large} with $H(x):=x-1-\log x$, and the second term by Chebyshev's inequality.
We now choose $c$ large enough that the right-hand side of (\ref{eq:large_dev}) above is at most $\delta/2$ for all $n$. Combining
%
(\ref{eq:large_dev}) with (\ref{eq:lilfin2}) and taking $C'=C_1+c$ (or slightly larger if necessary, to deal with small $n$) then establishes (\ref{eq:lilfin_aim}) and completes the proof. 
\end{proof}

Motivated by Corollary \ref{cor:lilfin}, we say that a path $P=(v_1,\ldots,v_{k+1})$ in $K_n$ is {\em $C$-legal} (for a given $C>0$), and write $P \in \scr{L}_C=\scr{L}_C(n)$, if for all 
$i=1,\ldots,k$, 
\begin{equation}\label{eq:legal_f}
\left|\frac{k}{w(P)}\cdot\sum_{j=1}^i X_{v_jv_{j+1}}-i\right| \leq C\sqrt{2i\log\log i},
\end{equation}
\emph{and}
\begin{equation}\label{eq:legal_b}
\left|\frac{k}{w(P)}\cdot\sum_{j=k-i+1}^k X_{v_jv_{j+1}}-i\right| \leq C\sqrt{2i\log\log i}.
\end{equation}
We need both conditions as we wish to ensure the path is well-behaved ``from both ends''. 
Though these conditions only restrict the weight of subpaths of $P$ starting from an end, they can be combined to 
control the behavior of other subpaths. 
More precisely, let 
\begin{equation}\label{eq:def_m}
m(i,k) = \min\{i-1,k+1-i\}\quad \mbox{and} \quad s(i,j,k) = \max\{m(i,k),m(j,k)\}.
\end{equation} Intuitively, 
$s(i,j,k)$ is the furthest distance from either $i$ or $j$ to one of the ends of the path. We then have 
\begin{lem}\label{lem:legal} If $P=(v_1,\ldots,v_{k+1})$ is $C$-legal then for all $1 \leq i < j \leq k+1$, 
\begin{equation}\label{eq:legal_m}
\left|\frac{k}{w(P)}\cdot\sum_{m=i}^{j-1} X_{v_mv_{m+1}}- (j-i)\right| \leq 2C \cdot \sqrt{2s(i,j,k) \log\log (s(i,j,k))}. 
\end{equation}
\end{lem}
\begin{proof}
Write $P(i)$ (respectively $P(i,j)$) for the path $(v_1,\ldots,v_i)$ (respectively $(v_{i},\ldots,v_j)$), then 
apply (\ref{eq:legal_f}) and (\ref{eq:legal_b}) to $P(i)$ or $P\setminus P(i)$ according as $i \leq (k+1)/2$ or $i > (k+1)/2$ and to $P(j)$ or $P\setminus P(j)$, similarly. 
Finally, use one of the equalities 
\[
P(i,j) = P(j)\setminus P(i-1) = (P\setminus P(i-1))\setminus (P\setminus P(j)) = P\setminus (P(i-1) \cup (P\setminus P(j))),
\]
together with the triangle inequality to prove the result.
\end{proof}

\subsection{Bounding the heights of shortest path trees}

We have now introduced all the technical apparatus we will need to control the path $P$ itself. Returning to our example from page \pageref{p:example}, we 
recall that our main goal is to bound the {\em heights of the shortest path trees leaving} $P$; Lemma \ref{lem:legal} tells us just how tall we can allow them to be. 
We say that $P$ is {\em $C$-bonsai} if:
\begin{itemize}\label{eq:star}
\item[($\star$)] for all $i=1,\ldots,k+1$, and all integers $\ell$ with $ \ell \geq m(i,k)/40$, in $K_n^P$ 
\[\spt_{v_i}(\ell + 2C\sqrt{500\ell\log\log \ell}) \mbox{ has height less than } \pran{\frac{9k}{10w(P)}}\cdot \ell.\]
\end{itemize}
We denote the set of $C$-bonsai paths in $K_n$ by $\scr{B}_C=\scr{B}_C(n)$. 
\begin{lem}\label{lem:key}
For any $C \geq 0$, if $w(P) \leq k \leq 4w(P)$ and $P$ is both $C$-legal and $C$-bonsai, i.e., $P\in \scr L_C \cap \scr B_C$, then $P$ is a local optimum.
\end{lem}

\begin{proof}[Proof of Lemma \ref{lem:key}]
Let $P=(v_1,\dots, v_{k+1})\in \scr L_C\cap \scr B_C$. 
Our aim is to show that for any integers $i,j$ with $1 \leq i < j \leq k+1$ and any path $Q$ of length $k$ whose intersection with $P$ is 
$(v_i,\ldots,v_{j})$, we have $w(Q) > w(P)$. 
By symmetry we may assume that $i < (k+1)/2$ (so $m(i,k)=i-1$) and that $m(i,k) \leq m(j,k)$ (so $s(i,j,k)=m(j,k)$). 

Let $Q_i$ (respectively $Q_j$) be the component of $Q\setminus P$ containing $v_i$ (respectively~$v_j$). 
We must have $|Q_i|+|Q_j|=k-(j-i) \geq m(i,k)+m(j,k)$. 
Let $k_i = (10/9)(w(P)/k)|Q_i|$ and define $k_j$ similarly. 
If $|Q_i| \geq k m(i,k)/(40 w(P))$ then since $P$ is $C$-bonsai and $w(P)/k \geq 1/4$,  
\[
w(Q_i) \geq k_i + 2C\sqrt{500 k_i\log\log k_i} \geq k_i + 2C\sqrt{m(i,k)\log\log m(i,k)}.
\]
(The constant $500$ ensures that the second inequality holds.)
Similarly, if $|Q_j| \geq (m(j,k)/40)k/w(P)$, then
\[
w(Q_j) \geq k_j + 2C\sqrt{500k_j\log\log k_j} \geq k_j + 2C\sqrt{m(j,k)\log\log m(j,k)}. 
\]
If both these conditions occur then by the definitions of $k_i$ and $k_j$ we have 
\[
w(Q_i)+w(Q_j) \geq \frac{w(P)}{k}(|Q_i|+|Q_j|)+2C\sqrt{s(i,j,k)\log\log s(i,j,k)}.
\]
Since $w(P)\leq k$, combining the preceding equation with the lower bound on 
$w(P \cap Q)$ given by Lemma \ref{lem:legal} 
completes the proof in this case.

On the other hand, if $|Q_j| \leq k m(j,k)/(40 w(P)) \leq m(j,k)/10$ then $|Q_i| \geq m(i,k)$ and
\[
|Q_i| \geq (k-(j-i)) - \frac{m(j,k)}{10} \geq \frac{9 (k-(j-i))}{10} \geq \frac{9m(j,k)}{10}.
\]
Since $P$ is $C$-bonsai we have $w(Q_i) \geq k_i + 2C\sqrt{500 k_i \log\log k_i}.$
Moreover, in this case, $k_i \geq (w(P)/k)(k-(j-i))$ and $k_i \geq m(j,k)/4=s(i,j,k)/4$, the preceding equation implies that 
\[
w(Q_i) \geq \frac{w(P)}{k}(k-(j-i)) + 2C\sqrt{s(i,j,k)\log\log s(i,j,k)},
\]
which, combined with Lemma \ref{lem:legal} and the fact that $w(P)\leq k$, 
completes the proof in this case. The case $|Q_i| \leq(m(i,k)/40)(k/w(P))$ 
is similar to but easier than the case $|Q_j|\leq (m(i,k)/40)(k/w(P))$.
\end{proof}

We are now in a position to fully define the special set of paths we wish to consider. 
Given $\epsilon>0$ and a positive integer $k$, let $\scr{U}_{\epsilon,k,C}=\scr{U}_{\epsilon,k,C}(n)$ be the set of paths 
$P$ of length $k$ in $K_n$ for which the following events hold
\begin{eqnarray}\label{eq:def_ai}
\samepage
A_1&:=&\{(1-2\epsilon) \log n \leq w(P) \leq (1-\epsilon)\log n\},\nonumber\\
A_2&:=&\{P \in \scr{P}^\star\} \\
A_3&:=&\{P \in \scr{L}_C \cap \scr{B}_C\}.\nonumber
\end{eqnarray} 
We remark that if $\epsilon>0$ is small enough (in particular, any $\epsilon=o(1)$ suffices), $C>0$ is fixed, $t=t(n)=o(\sqrt{\log n})$, and $P \in \scr{U}_{\epsilon,\lceil k_{\epsilon}+t\rceil,C}$ (where $k_\epsilon$ was defined in (\ref{eq:def_keps})) then 
for $n$ large enough, 
$w(P) \leq k \leq 4 w(P)$, so, by Lemma \ref{lem:key}, $P$ is a local optimum. 
The following lemma bounds the expected size of $\scr{U}_{\epsilon,k,C}$.
\begin{lem}\label{lem:u}
There exist $C>0$ and $\delta > 0$ such that for any fixed $c>2$ 
and any $\epsilon$ with $\delta > \epsilon > c\log\log n/\log n$, 
for any function $t(n)=o(\sqrt{\log n})$, and for all $n$ sufficiently large 
\[
\e{|\scr{U}_{\epsilon,\lceil k_{\epsilon}+t\rceil,C}|} \geq \delta \e{|\scr{P}_{\lceil k_{\epsilon}+t\rceil,\epsilon}|}.
\]
\end{lem}
\begin{proof}
Recall that $h(T)$ denotes the (unweighted) height of a rooted tree $T$.  
Fix $\epsilon>0$ as above, and a path $P$ of length $k=\lceil k_{\epsilon}+t \rceil $. 
It suffices to prove that for any such path $P$,
$\p{P \in \scr{U}_{\epsilon,\lceil k_{\epsilon}+t\rceil,C}}
\ge \delta \p{P\in \scr{P}_{\lceil k_{\epsilon}+t\rceil,\epsilon}}$ for
a suitable $\delta$.
Define the following conditional probability 
\[\psub{~\cdot~}:=\probC{~\cdot}{(1-2\epsilon) \log n \leq w(P) \leq (1-\epsilon)\log n}.\]
Then, by the definition (\ref{eq:def_ai}), we have 
\begin{eqnarray*}
\p{P \in \scr{U}_{\epsilon,\lceil k_{\epsilon}+t\rceil,C}}
& = & 
\p{A_1,A_2,A_3} \\
&=& \psub{A_2,A_3}\cdot \p{(1-2\epsilon) \log n\leq w(P) \leq (1-\epsilon)\log n} \\
& = &\psub{A_2,A_3}\cdot (1+o(1)) \p{w(P) \leq (1-\epsilon)\log n} \\*
& & \mbox{(by (\ref{eq:poisson_approx})} \\
& = &\psub{A_2,A_3}\cdot (1+o(1)) \p{P\in \scr{P}_{\lceil k_{\epsilon}+t\rceil,\epsilon}}~.
\end{eqnarray*} 
So to prove the lemma it suffices to show that $\psub{A_2,A_3}$ is bounded away from zero. To do so we first write
\begin{eqnarray*}
\psub{A_2,A_3} 	&=& \psub{P \in \scr{P}^\star,P\in \scr{L}_C,P \in \scr{B}_C} \\
				&\geq& \psub{P \in \scr{B}_C,P \in \scr{L}_C} - \psub{P \notin \scr{P}^\star}.
\end{eqnarray*}
For $\epsilon$ as above, $\psub{P \notin \scr{P}^\star} \rightarrow 0$ as $n \rightarrow \infty$ by Corollary \ref{cor:light_means_lightest}, 
so we need only prove that $\psub{P \in \scr{B}_C,P \in \scr{L}_C}$ is bounded away from zero. To do so, rather than considering  
the set $\scr{B}_C$ we consider a slightly different collection, which we denote $\scr{B}_C^{\epsilon}$. 
We let $P \in \scr{B}_C^{\epsilon}$ if ($\star$) page \pageref{eq:star} holds with the term $w(P)$ replaced by $(1-\epsilon)\log n$. 
Given that $w(P) \leq (1-\epsilon)\log n$, we have $\scr{B}_C^{\epsilon} \subseteq \scr{B}_C$, so 
$\psub{P \in \scr{B}_C,P \in \scr{L}_C} \geq \psub{P \in \scr{B}_C^{\epsilon},P \in \scr{L}_C}$. 
But the event $P \in \scr{B}_C^{\epsilon}$ is independent from the events $P \in \scr{L}_C$ and
\[
A_1=\{(1-2\epsilon)\log n \leq w(P) \leq (1-\epsilon)\log n\},
\] 
since they are determined by disjoint sets of edges. Thus, we have 
\[
\psub{P \in \scr{B}_C,P \in \scr{L}_C} \geq \p{P \in \scr{B}_C^{\epsilon}}\psub{P \in \scr{L}_C}. 
\]
Now choose $C=C'(1/2)$, where $C'(1/2)$ is as in Corollary \ref{cor:lilfin}. By that corollary we then have 
$\psub{P \in \scr{L}_C} \geq 1/2$, so 
\[
\psub{P \in \scr{B}_C,P \in \scr{L}_C} \geq \frac{1}{2} \p{P \in \scr{B}_C^{\epsilon}}. 
\]
To complete the proof, it thus suffices to prove that $\p{P \in \scr{B}_C^{\epsilon}}$ is bounded below by a constant. 

We assume 
$\epsilon$ is small enough and $n$ is large enough (say $n \geq n_0$ for some $n_0$) that
\begin{equation}\label{eq:epslarge} 
\frac{(1-\epsilon)\log n}{\lceil k_{\epsilon}+t\rceil} \leq  \frac{1}{3.5}. 
\end{equation} 
(Note that this is possible by the value of $\alpha^\star$. We also remark that the preceding equation is the only place we use the hypothesis that $\epsilon$ is 
``small''.)
To further simplify our lives, we now consider a subset of $\scr{B}_C^{\epsilon}$ that we denote 
$\bigcap_{i=1}^{k+1} \scr{B}_{C,i}$, where $\scr{B}_{C,i}$ is the set of all paths $Q=(v_1,\ldots,v_{k+1})$ 
in $K_n$ for which, for all integers $\ell$ with $\ell \geq m(i,k)/40$ (recall that $m(i,k)$ is defined in (\ref{eq:def_m})), in $K_n^Q$ 
\begin{equation}\label{eq:2star}
h\pran{\spt_{v_i}\pran{\ell+2C\sqrt{500\ell\log\log \ell}}}< 3.15 \ell.
\end{equation}
The set $\bigcap_{i=1}^{k+1} \scr{B}_{C,i}$ is indeed a subset of $\scr{B}_C^{\epsilon}$, since 
$9k/(10(1-\epsilon)\log n) \geq 3.15$ by (\ref{eq:epslarge}). 
The events $P \in \scr{B}_{C,i}$, $i=1,\ldots,k+1$ are all increasing in the weight values, so by the FKG inequality \cite{harris1960lbc,alon00proba,grimmett99percolation, fortuin1971cip}, we have
\begin{equation}
\p{P \in \scr{B}_C^{\epsilon}} \geq \prod_{i=1}^{k+1} \p{P \in \scr{B}_{C,i}} \geq \prod_{i=1}^{\lfloor (k+1)/2\rfloor} \p{P \in \scr{B}_{C,i}}^2,\label{eq:bceps}
\end{equation}
where the second inequality takes advantage of the symmetry.
For small $\ell$, the probability that (\ref{eq:2star}) occurs is bounded away from zero:
it is at least the probability that all $n-1$ edges leaving $v_i$ have weight at least $\ell+2C\sqrt{500\ell\log\log \ell}$, and every edge weight has an exponential distribution with mean $n$. 
Thus, choose an integer $\ell_0$ large enough that the following four conditions are satisfied for $\ell \geq \ell_0$, 
\[
\left\{
\begin{array}{l}
\ell+2C\sqrt{500\ell\log\log \ell} \leq 1.05 \ell,\\
1-e^{-\ell/10}/2 \geq \exp\pran{-e^{-\ell/10}}
\end{array}
\quad\mbox{and}\quad
\left\{
\begin{array}{c l l}
e^{3-\ell/8} &\leq& e^{-\ell/10}/4\\
3 e^{\ell/40} \exp\pran{-e^{\ell/20}} &\leq& e^{-\ell/10}/4.
\end{array}
\right.
\right. 
\] 
Then there is $\alpha_0>0$ such that for $\ell \leq \ell_0$, 
the event in (\ref{eq:2star}) occurs with probability at least $\alpha_0$. 
Furthermore, for each $\ell$ the event in (\ref{eq:2star})
is increasing in the weight values, so the FKG inequality yields 
\begin{eqnarray}
\p{P \in \scr{B}_{C,i}} 
&\geq & \prod_{\ell=\lceil m(i,k)/40\rceil}^{\infty} \p{h\pran{\spt_{i}\pran{\ell+2C\sqrt{500\ell\log\log \ell}}}< 3.15 \ell} \nonumber\\
& \geq& \alpha_0^{\xi_i} \cdot \prod_{\ell\ge \ell_0\vee m(i,k)/40}  \p{h\pran{\spt_i\pran{1.05\ell}}< 3.15 \ell}, \label{eq:bci}
\end{eqnarray}
where we have written $\xi_i=\max\{\ell_0-m(i,k)/40, 0\}$.
We now turn to bounding these remaining probabilities for $\ell \geq \ell_0$. 
By Lemma~\ref{lem:sptheight}, and our choice for $\ell_0$,
\[
\p{|\spt_i\pran{1.05\ell}|\geq e^{1.1 \ell}} \leq 3 e^{\ell/40} \exp\pran{-e^{\ell/40}} \leq \frac{e^{-\ell/10}}4. 
\]
Furthermore, with the following values and our choice for $\ell_0$, the bound of Theorem \ref{thm:luc} yields
\[
\probC{h\pran{\spt_{i}\pran{1.05\ell}}\geq 3.15 \ell}{|\spt_i\pran{1.05\ell}|< e^{1.1 \ell}} 
 \leq \frac{e^{-\ell/10}}4. 
\]
It follows from the two preceding equations that for $\ell \geq \ell_0$
\[
\p{h\pran{\spt_{i}\pran{1.05\ell}}< 3.15 \ell} \geq \pran{1-\frac{e^{-\ell/10}} 4}^2\ge 1-\frac{e^{-\ell/10}}2 \geq e^{-e^{-\ell/10}}. 
\]
Combining this inequality with (\ref{eq:bci}) yields
\begin{eqnarray*}
\p{P \in \scr{B}_{C,i}} 
\geq  \alpha_0^{\xi_i} \exp\pran{-\sum_{\ell=\lceil m(i,k)/40\rceil}^{\infty} e^{-\ell/10}}
				\geq  \alpha_0^{\xi_i} \exp\pran{-26e^{- m(i,k)/400}}.
\end{eqnarray*}
Observe that the multiplicative factor $\alpha_0^{\xi_i}=1$ but for the values of $i$ such that $m(i,k)\le 40 \ell_0$. Observe also that for $i\le (k+1)/2$, we have by definition (\ref{eq:def_m}) $m(i,k)=i-1$. So, together with (\ref{eq:bceps}) the preceding equation yields
\[
\p{P \in \scr{B}_{C}^{\epsilon}}  \geq  \alpha_0^{80\ell_0^2} \exp\pran{-52 \sum_{i=1}^{\lfloor (k+1)/2\rfloor} e^{- (i-1)/400}},
\]
which completes the proof, since the series in the right-hand side converges.
\end{proof}

With Lemma \ref{lem:u} under our belt, we can now complete our proof of Theorem~\ref{thm:main}

\begin{proof}[Proof of Theorem \ref{thm:main}, lower bounds]
Let $C$ be as in Lemma \ref{lem:u}, choose $\epsilon=\epsilon(n) = 3\log \log n/\log n$, and choose a function $t=t(n)=o(\sqrt{\log n})$ 
such that $t(n) \rightarrow -\infty$ with $n$; we will make a more precise choice for $t(n)$ shortly. 
As noted just before the statement of Lemma \ref{lem:u}, 
if $\epsilon=o(1)$, $C>0$ is fixed, $t=o(\sqrt{\log n})$, and $P \in \scr{U}_{\epsilon,\lceil k_{\epsilon}+t\rceil,C}$ then 
for $n$ large enough, $w(P) \leq k \leq 4 w(P)$, so, by Lemma \ref{lem:key}, $P$ is a local optimum. 
Our choices of $\epsilon, C$, and $t$ satisfy these conditions, so for $n$ large enough we have 
\[
\scr{U}_{\lceil k_{\epsilon}+t\rceil ,\epsilon,C} \subseteq \scr{O}_{\lceil k_{\epsilon}+t\rceil ,\epsilon} \subseteq \scr{P}_{\lceil k_{\epsilon}+t\rceil,\epsilon}.
\]
By Lemmas \ref{eq:exact_formula} and \ref{lem:u}, and the definition
of $k_{\epsilon}$ in (\ref{eq:def_keps}), there is $K>0$ such that choosing $t(n)=K\log\log
n$, for $n$ large enough we have
\[
\e{|\scr{U}_{\lceil k_{\epsilon}+t\rceil,\epsilon}|} \geq \log n.
\]
Recalling the definition (\ref{eq:smmcount}) of $\Delta(\scr{U}_{\lceil k_{\epsilon}+t\rceil,\epsilon})$, we remark that since 
$\scr{U}_{\lceil k_{\epsilon}+t\rceil ,\epsilon,C} \subseteq \scr{O}_{\lceil k_{\epsilon}+t\rceil ,\epsilon}$, the contribution of 
pairs of paths who intersect once to $\Delta(\scr{U}_{\lceil k_{\epsilon}+t\rceil,\epsilon})$, is zero. 
Also, since
$\scr{U}_{\lceil k_{\epsilon}+t\rceil ,\epsilon,C}\subseteq \scr{P}_{\lceil k_{\epsilon}+t\rceil,\epsilon}$
, we have 
\[
\Delta(\scr{U}_{\lceil k_{\epsilon}+t\rceil,\epsilon}) \leq \sum_{2 \leq j \leq i < \lceil k_{\epsilon}+t\rceil} \Delta_{t,i,j},
\]
so by Lemma \ref{eq:smm_bound}, for $n$ large enough, we have $\Delta(\scr{U}_{\lceil k_{\epsilon}+t\rceil,\epsilon}) \leq n^{-0.95} \E{|\scr{P}_{\lceil k_{\epsilon}+t\rceil,\epsilon}|}^2.$ 
By our choice of $\epsilon$ there is $K'>0$ such that $k_{\epsilon}+t \geq \alpha^\star\log n - K'\log\log n$. 
It follows by the second moment method (\ref{eq:smm}) and Lemma~\ref{lem:u} that 
\begin{eqnarray*}
\p{\max_{1 \leq i,j \leq n}|P_{ij}| \leq \alpha^\star \log n - K' \log\log n} & \leq & \p{|\scr{U}_{\lceil k_{\epsilon}+t\rceil,\epsilon}|=0} \\
& <& \frac 1 {\e|\scr{U}_{\lceil k_{\epsilon}+t\rceil,\epsilon}|}+ \frac{\Delta(\scr{U}_{\lceil k_{\epsilon}+t\rceil,\epsilon})}{\E{|\scr{U}_{\lceil k_{\epsilon}+t\rceil,\epsilon}|}^2}\\
&\le& \frac 1{ \log n} + \frac{n^{-0.95}}{\delta^2} \xrightarrow[n\to\infty]{} 0,
\end{eqnarray*}
proving the lower bound in probability. We then have from the previous bound 
\[
\E{\max_{1 \leq i,j \leq n}|P_{ij}|} 
 \geq (\alpha^\star \log n - K' \log\log n)\pran{1-\frac 1 {\log n} - \frac{n^{-0.95}}{\delta^2}},
\]
which proves the lower bound in expectation.
\end{proof}

\section{Conclusion} \label{sec:concl}

We have shown that the longest minimum-weight path in a complete graph
with i.i.d.\ exponentially distributed edge weights is asymptotically
$\alpha^\star \log n$ in probability, where $\alpha^\star\approx 3.5911$
satisfies $\alpha^\star \log \alpha^\star = \alpha^\star + 1$. The difficult part of the proof
is the lower bound as the upper bound follows easily from Devroye's
\cite{devroye87branching} results on the height of a random recursive
tree and a simple union bound. 
The fact that such a simple union bound leads to correct bounds means
that the $n$ shortest path trees $\spt_i$ essentially behave as if
they were independent.

The lower bound is proved by a second-moment argument for a carefully 
chosen random set of minimum-weight paths, all of length 
about $\alpha^\star \log n$, that satisfy certain regularity conditions.
This careful choice allows us to control the variance of the size 
of the set of such paths. The proof shows that, with high
probability, there exist minimum-weight paths of length 
about $\alpha^\star \log n$ whose total weight is close to $\log n$. 
Thus, these extra-long paths have a total weight just like a 
typical minimum-weight path.

Finally, just as \citet{janson99one} did in proving his results discussed in the introduction, 
we can use a standard coupling argument to extend Theorem \ref{thm:main} to distributions other than exponential. 
To explain the coupling, it is useful to divide the edge weights by $n$, so to consider the 
$\expo(1)$ edge weights $X_e'=X_e/n$. 

Let $Z$ be any non-negative random variable, with distribution function 
$F(t) = \p{Z \leq t}$, and let $F^{-1}(t) = \sup\{x: F(x)\leq t\}$.  
Let $U$ be a random variable uniform on $[0,1]$. Then $F^{-1}(U)$ is distributed 
as $Z$. We may thus couple the exponential edge weights $X_e'$ to uniform 
edge weights $U_e$ by setting $X_e' = F_E^{-1}(U_e)$, where $F_E(t)$ is the 
distribution function of an $\expo(1)$ random variable. We 
have $F_E(t) = 1-e^{-t} = t -t^2/2 + O(t^3)$ as $t \rightarrow 0$, and it follows that for all 
edges $e$ of $K_n$, $|X_e'-U_e|=O((X_e')^2)$, uniformly when $X_e'\le 1/2$. 

Of particular use for us is the following consequence: for all edges $e$ 
with $X_e' \leq 12 \log n/n$ (say), we have $|X_e'-U_e| = O(\log^2 n/n^2)$, 
so for any {\em path} $P$ with $O(\log n)$ edges and with $w'(P):=\sum_{e \in E(P)} X_e' \leq 12 \log n/n$, 
we have 
\begin{equation}\label{coupling}
\left|w'(P)-\sum_{e \in E(P)} U_e\right| = O\pran{\frac{\log^3}{n^2}}. 
\end{equation}
It then follows easily that the conclusions of Theorem \ref{thm:main} hold 
for uniform $[0,1]$, as well as exponential, edge weights. (Technically, one must 
redo all the arguments of the paper keeping (\ref{coupling}) in mind, and accordingly adjust some of the 
constants very slightly; we omit this step.) 

In fact all we really need is a bound of $o(1/n)$ in (\ref{coupling}) in order to extend Theorem \ref{thm:main} 
to uniform $[0,1]$ edge weights. Thus, by another coupling as above, we may extend Theorem \ref{thm:main} from uniforms to 
any other edge weights $Y$ whose distribution function $F_Y$ satisfies 
$F_Y(t)=ct+o(t/\log^2 t)$ for some constant $c> 0$, as 
$t \rightarrow 0$. 

\setlength{\bibsep}{.2em}

\end{document}